\documentclass[11pt]{amsart}
\usepackage{amsmath,amssymb,amscd}
\usepackage{mathpazo}
\begin{document}

\newcommand{\GL}{\operatorname{GL}}

\newcommand{\CC}{\mathbb C}
\newcommand{\NN}{\mathbb N}
\newcommand{\QQ}{\mathbb Q}
\newcommand{\RR}{\mathbb R}
\newcommand{\ZZ}{\mathbb Z}

\newcommand{\Qp}{{ {\mathbb Q}_p} }
\newcommand{\Zp}{{{\mathbb Z}_p}}

\newcommand{\cF}{{\mathcal F}}
\newcommand{\cL}{{\mathcal L}}

\newcommand{\bz}{{\mathbf 0}}

\newcommand{\sgn}{\operatorname{sgn}}

\theoremstyle{theorem}
\newtheorem{thm}{Theorem}[section]
\newtheorem{lem}[thm]{Lemma}
\newtheorem{prop}[thm]{Proposition}
\newtheorem{conj}[thm]{Conjecture}
\newtheorem{fact}[thm]{Fact}
\newtheorem{cor}[thm]{Corollary}

\theoremstyle{remark}
\newtheorem{Rk}[thm]{Remark}

\theoremstyle{definition}
\newtheorem{Def}[thm]{Definition}
\newtheorem{convention}[thm]{Convention}
\newtheorem{notation}[thm]{Notation}

\title{A Euclidean Skolem-Mahler-Lech-Chabauty method}
\author{Thomas Scanlon}
\thanks{Partially supported by NSF grants FRG DMS-0854998 and DMS-1001550.}
\email{scanlon@math.berkeley.edu}
\address{University of California, Berkeley \\
Department of Mathematics \\
Evans Hall \\
Berkeley, CA 94720-3840 \\
USA}

\begin{abstract}
Using the theory of o-minimality we show that the $p$-adic method of Skolem-Mahler-Lech-Chabauty may be adapted to prove instances of the dynamical Mordell-Lang conjecture for some real analytic dynamical systems.  For example, we show that if $f_1, \ldots, f_n$ is a finite sequence of real analytic functions $f_i:(-1,1) \to (-1,1)$ for which $f_i(0) = 0$ and $|f_i'(0)| \leq 1$ (possibly zero), $a = (a_1,\ldots,a_n)$ is an $n$-tuple of real numbers close enough to the origin and $H(x_1,\ldots,x_n)$ is a real analytic function of $n$ variables, then the set $\{ m \in \NN : H (f_1^{\circ m} (a_1), \ldots, f_n^{\circ m}(a_n)) = 0 \}$ is either all of $\NN$, all of the odd numbers, all of the even numbers, or is finite.
\end{abstract}

\maketitle

\section{Introduction}

Consider the following form of the dynamical Mordell-Lang conjecture enunciated in~\cite{GT} generalizing Zhang's version from~\cite{Zh}.

\begin{conj}
\label{DMLconj}
Let $K$ be a field of characteristic zero, $X$ an algebraic variety over $K$, $f:X \to X$ a regular self-map of $X$ also defined over $K$, $a \in X(K)$ a $K$-rational point and $Y \subseteq X$ a closed subvariety.  Then the set $\{ n \in \NN : f^{\circ n} (a) \in Y(K) \}$ is a finite union of arithmetic progressions (where we allow the modulus of an arithmetic progression to be zero so that a singleton is an arithmetic progression).
\end{conj}

By adapting Skolem's $p$-adic method~\cite{Sk} (attributed to and extended and developed by, at least, Lech~\cite{Le}, Mahler~\cite{Mahler}, and Chabauty~\cite{Chab} as well) for analyzing algebraic relations on cyclic subgroups of algebraic groups to more general algebraic dynamical systems, one might hope to prove Conjecture~\ref{DMLconj}, and, indeed, in~\cite{BGKT} exactly this strategy is employed.  Since all of the data are defined over a finitely generated subfield of $K$, we may assume that $K$ itself is finitely generated.   Under suitable hypotheses, we can find some embedding $K \hookrightarrow \Qp$ so that there is a $p$-adic analytic function $F:\Zp \to X(\Qp)$ interpolating the function $n \mapsto f^{\circ n}(a)$ on $\NN$.  The set $\{ x \in \Zp : F(x) \in Y(\Qp) \}$ is then defined by the vanishing of a $p$-adic analytic function and as such is a finite union of points and cosets of $p^m \Zp$ for some $m \geq 0$.  Hence, $\{ n \in \NN : f^{\circ n}(a) \in Y(K) \}$ is also a finite union of points and arithmetic progressions with modulus $p^m$.

On the face of it, it would seem that the argument sketched in the above paragraph would yield no information if $\Qp$ were replaced by $\RR$.  Indeed, if $S \subseteq \NN$ were any set of natural numbers, then we could find a real analytic function which vanishes exactly on $S$.  However, the point of this note is that in many cases of interest, including some which are not amenable to the $p$-adic method due to the presence of superattracting points, one may find a function $F:\RR_+ \to X(\RR)$ interpolating $n \mapsto f^{\circ n} (a)$ which is definable in an \emph{o-minimal} expansion of the field of real numbers.  It then follows directly from o-minimality that if such an orbit has infinite intersection with an algebraic variety, then all but finitely many of the point from that orbit lie in the subvariety.

We recognize that the notions of definability and o-minimality may not be familiar to the reader approaching this note from outside of logic.  The recent work of Pila on the Andr\'{e}-Oort conjecture which uses an analysis of definability in o-minimal structures in essential ways may have introduced these ideas to number theorists, and in Section~\ref{cr} we speculate about how the estimates of~\cite{PW} may give useful information about higher rank dynamical Mordell-Lang problems.  In any case, in Section~\ref{omin} we recall the basic definitions and results about the specific o-minimal structures, of the real field considered together with the real exponential function and restricted analytic functions, that we shall require.

\section{O-minimality}
\label{omin}

The basic theory of o-minimality is exposed in~\cite{vdD} and the proofs that the structure we shall use, namely the field of real numbers endowed with the exponential function and bounded analytic functions, is o-minimal can be found in~\cite{DM,DMM}.  We recall here some fundamental definitions and results.

\begin{Def}
A structure $(M,<,\ldots)$ which is totally ordered by the relation $<$ is \emph{o-minimal} if every definable (with parameters) subset of $M$ is a finite union of singletons and intervals of the form  $(-\infty,a) := \{ x \in M : x < a \}$, $(a,b) := \{ x \in M : a < x < b \}$ or $(b,\infty) := \{ x \in M : b < x \}$ for some $a, b \in M$.
\end{Def}

Here \emph{structure} is understood in the sense of first-order logic.  That is, the set $M$ is endowed not only with an ordering interpreting $<$ but it also interprets the extra function, relation and constant symbols elided by the ellipses.  Likewise, to say that a set $X \subseteq M$ is \emph{definable (with parameters)} is to say that there is a first-order formula $\phi(x)$ in one free variable $x$ in the language of $M$ possibly augmented by constant symbols naming elements of $M$ so that $$X =\{ a \in M :  \text{ the formula } \phi \text{ is true in } M \text{ when }a\text{ is substituted for }x \}$$

While the theory as a whole is enriched through the study of general o-minimal structures especially through a kind of nonstandard analysis, our direct application of o-minimality concerns only the real numbers and the extra structure is specified by naming some functions.  Thus, rather than rehashing the theory of definability in general structures, we restrict to the case of expansions of the ordered field of real numbers by families of functions.

\begin{Def}
Suppose that for each $n \in \ZZ_+$ that $\cF_n$ is a set of real valued functions $f:\RR^n \to \RR$ of $n$ real variables.  By $\cL_\cF$ we mean the first-order language having a binary relation symbol $\leq$, binary function symbols $+$ and $\cdot$, constant symbols $r$ for each $r \in \RR$, and $n$-ary function symbols $f$ for each $f \in \cF_n$.  By $\RR_\cF$ we mean the $\cL_\cF$-structure having universe $\RR$ on which the nonlogical symbols are interpreted eponymously.
\end{Def}

In the structure $\RR_\cF$, sets of the form $\{ (a_1,\ldots,a_n) \in \RR^n : f(a_1,\ldots,a_n) \leq g(a_1,\ldots,a_n) \}$ are definable where $f, g \in \cF_n$ or more generally where $f$ and $g$ are obtained from the projection functions, constant functions, functions in $\cF$ and addition and multiplication via appropriate compositions.    We obtain the class of quantifier-free definable sets by closing off under finite Boolean operations.  In general, as first-order logic permits the application of existential quantifiers or what is is the same thing at the level of the definable sets themselves, images under coordinate projections, there will be definable sets which cannot be expressed in the simple form of finite Boolean combinations of sets defined by inequalities between the basic functions.   It is a celebrated theorem of Tarski~\cite{Tar} that when $\cF = \varnothing$, so that we are consider the set of real numbers simply as an ordered field, every definable set in any number of variables is quantifier-free definable.   In this case, the functions $f$ and $g$ are simply polynomials over the real numbers and the definable sets are the \emph{semi-algebraic} sets, sets of $n$-tuples of real numbers defined by polynomial inequalities.  Specializing to $n = 1$, using the fact that nonconstant polynomials have only finitely many zeros and can change sign only at their zeros, we see that $\overline{\RR} := \RR_{\varnothing} = (\RR,\leq,+,\cdot)$ is o-minimal.  Wilkie showed~\cite{Wilkie} using results of Khovanski on few-nomials~\cite{Kh} that when $\cF = \{ \exp \}$, then every definable set may be expressed as a projection of a basic set and that $\RR_{\exp} := \RR_{ \{ \exp \} } = (\RR,\leq,+,\cdot,\exp)$ is o-minimal.

It should be clear that we cannot take $\cF$ to consist of all real analytic functions and hope for $\RR_\cF$ to be o-minimal, as, for instance, the sine function is globally analytic but the infinite discrete set $\{ x \in \RR : 0 = \sin(x) \}$ is definable in $\RR_{\{ \sin \}}$.    However, if we consider only \emph{restricted} analytic functions, then the resulting structure is o-minimal.  That is, we let $\cF_n$ consist of all function $f:\RR^n \to \RR$ for which \begin{itemize}
\item there is some neighborhood $U \supseteq [-1,1]^n$ and a real analytic function $g:U \to \RR$ for which $g \upharpoonright [-1,1]^n = f \upharpoonright [-1,1]^n$ and
\item $f(x_1,\ldots,x_n) = 0$ if $|x_i| > 1$ for any $i \leq n$.
\end{itemize}

\begin{Rk}
There is more than one reasonable way to formalize the idea of including all restricted analytic functions.  For example, one might want to allow for restrictions of analytic functions to other polyhedra or the convention that the function is extended by zero outside of the box on which it is analytic might be replaced by the convention that the function is simply undefined in that region.  Ultimately, these variants lead to the same class of definable sets.
\end{Rk}

In parallel with their work on $p$-adic analytic function, Denef and van den Dries~\cite{DD} proved a quantifier elimination theorem for restricted real analytic functions from which one deduces that $\RR_{an} := \RR_\cF$ is o-minimal.

Finally, for us but not in the search for interesting and useful o-minimal structures, the structure obtained by enriching $\RR_{an}$ with the global exponential function, $\RR_{an,\exp}$, is itself o-minimal~\cite{DM,DMM}.

With our application to the problem of describing intersections of dynamical orbits with algebraic varieties, nothing more than the definition of o-minimality and the fact that $\RR_{an,\exp}$ is o-minimal will be used.  However, deeper consequences of o-minimality may be relevant to the questions about orbits with respect to higher rank monoids of operators which we shall discuss in Section~\ref{cr}.

\section{Main theorem}

In this section we observe that in many cases of interest with weaker hypotheses than what are required for linearization some dynamical orbits are o-minimally uniformized from which we deduce these instances of the dynamical Mordell-Lang conjecture.

Let us begin with a lemma on exponentiation in linear groups.    Before we do so, we set some notation.

\begin{notation}
For $n \in \ZZ_+$ a positive integer we write $\GL_n^+(\RR)$ for the set of invertible $n \times n$ real matrices all of whose eigenvalues are real and positive. We write $I_n \in \GL_n(\RR)$ for the identity.  We let $J_n \in \GL_n(\RR)$ be the matrix with $1$s along the upper off diagonal and zeroes everywhere else.  That is, $$(J_n)_{i,j} = \begin{cases} 1 \text{ if } j = i+1 \\ 0 \text{ otherwise } \end{cases}$$
For $n_1, \ldots, n_\ell$ a sequence of positive integers, there is a natural inclusion of groups $\GL_{n_1}(\RR) \times \cdots \times \GL_{n_\ell}(\RR) \hookrightarrow \GL_{n_1 + \cdots + n_\ell}(\RR)$.   We denote the image of $(g_1,\ldots,g_\ell) \in \GL_{n_1}(\RR) \times \cdots \times \GL_{n_\ell}(\RR)$ under this map by $g_1 \oplus \cdots \oplus g_\ell$.

By a \emph{partition} of a positive integer $n$ we mean a finite non-decreasing sequence $\pi = (\pi_1,\ldots,\pi_\ell)$ of positive integers for which $n = \sum_{j=1}^\ell \pi_i$.
\end{notation}

\begin{lem}
\label{expgln}
For each positive integer $n \in \ZZ_+$ there is a unique function $E_n:\RR \times \GL_n^+(\RR) \to \GL_n^+(\RR)$ definable in $\RR_{\exp}$ which satisfies the condition that for $g \in \GL_n^+(\RR)$ and $x \in \RR$ we have $E_n(1,g) = g$ and $E_n(x+1,g) = g \cdot E_n(x,g)$.
\end{lem}

\begin{proof}
Let us prove uniqueness.   Suppose that $F_n$ is another function definable in $\RR_{\exp}$ which satisfies the stated difference equation.  Let $g \in \GL_n^+(\RR)$.  Since the set $S := \{ x \in \RR : E_n(x,g) = F_n(x,g) \}$ is definable in the o-minimal structure $\RR_{\exp}$, there must be some $B \in \RR_+$ so that either $(B,\infty) \subseteq S$ or $(B,\infty) \cap S = \varnothing$.   As the difference equation implies that $E_n(m,g) = g^m = F_n(m,g)$ for $m \in \ZZ$, we see that $(B,\infty) \subseteq S$.  Applying the difference equation again we see that $S = \RR$, as claimed.

Let us now define $E_n$.

\begin{eqnarray*}
E_n(x,g) = y & \Longleftrightarrow & \bigvee_{(\pi_1, \ldots, \pi_\ell) \text{ a partition of }n}
(\exists h \in \GL_n(\RR))(\exists \lambda_1, \ldots,\lambda_\ell \in \RR_+)  \\
&& \bigwedge_{\pi_i = \pi_{i+1}} \lambda_i \geq \lambda_{i+1} \& \\
& & h g h^{-1} = (\lambda_1 I_{\pi_1} + J_{\pi_1}) \oplus \cdots \oplus (\lambda_\ell I_{\pi_\ell} + J_{\pi_\ell})  \& \\
&& h y h^{-1} = \exp(x \ln(\lambda_1)) \sum_{j=0}^{\pi_1 - 1} \binom{x}{j} \lambda_1^{-j} J_{\pi_1}^j \oplus \cdots \\
&& \oplus \exp(x \ln(\lambda_\ell)) \sum_{j=0}^{\pi_\ell - 1} \binom{x}{j} \lambda_\ell^{-j} J_{\pi_1}^j
\end{eqnarray*}

While we have employed standard mathematical abbreviations in the formula defining $E$ (for example, to speak of $\ln(\lambda_i)$ we should really quantify over a new variable $z_i$ and include the defining condition $\exp(z_i) = \lambda_i$), it should be clear that the expression on the right may be given by a formula in the language of $\RR_{\exp}$ and that it defines a \emph{relation} between $(x,g)$ and $y$, the purported value of $E_n(x,g)$.  That $y$ is a function of $(x,g)$ is a consequence of two facts.  First, the Jordan form of a matrix is unique provided that we normalize (as we have) so that size of the blocks is non-increasing and then within the blocks of a given size the eigenvalues are non-increasing.  Secondly, if we were to choose $h'$ to be another matrix which conjugates $g$ to its Jordan form, then $h' = k h$ where $k$ centralizes the Jordan matrix.  It then follows that $k$ centralizes any integer power of the Jordan matrix and consequently, by o-minimality again, any real power.
\end{proof}

Most everything we discuss will make sense only in some neighborhood of the function under consideration.  However, we do not wish to speak about germs as we will actually apply these functions.  Instead we shall employ expression like `` $\Phi$ is a self-map near the origin'' or that some compositional identity holds near the origin to mean, in the former case, that there is some open set $U \subseteq \RR^n$ with $\bz = (0,\ldots,0) \in U$ for which $\Phi:U \to U$ and $\Phi(\bz) = \bz$, while in the latter case we mean that the identity in question holds on a neighborhood of the origin.  When the ambient dimension $n$ is relevant, it will be mentioned explicitly.  Likewise, we may speak of some point $a$ being ``close enough to the origin'' by which we mean that $a$ belongs to a neighborhood of the origin with respect to which the mentioned self-maps restrict to self-maps and the pertinent compositional identities hold.

\begin{Def}
We say that a real analytic self-map $\Phi$ near the origin is \emph{projectively linearizable} if there is a real analytic function $\alpha$ which fixes the origin and is invertible near the origin for which $\alpha \circ \Phi \circ \alpha^{-1}$ is given by a fractional linear transformation near the origin.  We say that $\Phi$ is \emph{strongly projectively linearizable} if moreover every eigenvalue of some matrix representing its projective linearization is real and positive.

We say that $\Phi$ is \emph{monomializable} if $d \Phi_\bz \equiv \bz$ and there are a real analytic $\alpha$ as above, a matrix $M \in \GL_n(\QQ)$ all of whose entries are positive integers but not having roots of unity amongst its eigenvalues, and a tuple $\lambda = (\lambda_1,\ldots,\lambda_n) \in (\RR^\times)^\ell$ so that near the origin

$$\alpha \circ \Phi \circ \alpha^{-1} (x) = \lambda \cdot x^M = (\lambda_1 x_1^{M_{1,1}} x_2^{M_{1,2}} \cdots x_n^{M_{1,n}}, \ldots, \lambda_n x_1^{M_{n,1}} x_2^{M_{n,2}} \cdots x_n^{M_{n,n}})$$

We say that $\Phi$ is \emph{strongly monomializable} if moreover every eigenvalue of $M$ is real and positive.

\end{Def}

\begin{Rk}
The restriction that powers of $M$ not have a nontrivial fixed vector might seem unnatural as, of course, for instance, the polynomial $x_1$ is certainly a monomial, but since linear maps and higher degree monomials have different behaviors, we have imposed this condition.
\end{Rk}

\begin{Rk}
One might wish to identify linearizability  in which $\Phi$ is analytic conjugate to its differential at the origin as a separate case from projective linearizability.  However, for purposes of our arguments there is nothing to be gained from such a separation.
\end{Rk}

\begin{lem}
If $\Phi$ is monomializable, then one may choose the vector $\lambda = (\lambda_1, \ldots, \lambda_n)$ to consist entirely of $\pm 1$
\end{lem}
\label{pm1}
\begin{proof}
By hypothesis, we have conjugated $\Phi$ to $\lambda \cdot x^M$ for some $\lambda \in (\RR^\times)^n$ and $M \in \GL_n(\RR) \cap \operatorname{Mat}_{n \times n}(\NN)$.  Let us consider the result of conjugating by the action of a linear map $x \mapsto \mu \cdot x$ where $\mu = (\mu_1,\ldots,\mu_n) \in (\RR^\times)^n$.  We compute that $$\mu^{-1} \cdot (\lambda \cdot (\mu \cdot x)^M = \mu^{M - 1} \cdot \lambda \cdot x^M$$  Since $1$ is not an eigenvalue of $M$, the matrix $M - 1$ has full rank.  Hence, the map $x \mapsto x^{M-1}$ is a bijective self-map of $(\RR_+)^n$ and in particular we can find $\mu$ so that $\mu^{M-1} = (1/|\lambda_1|, \ldots, 1/|\lambda_n|)$.
\end{proof}

\begin{prop}
\label{monoexp}
Given a dimension $m$ there is another number $B = B(m)$ so that for any strongly monomializable $\Phi$ in $m$ variables and any point $a$ close enough to the origin and in its $\Phi$-attracting basin there is a function $F:[0,\infty) \to \RR^m$ definable in $\RR_{an,\exp}$ satisfying $E(0) = a$ and the functional equation $F(x+1) = \Phi^{\circ B}(F(x))$.
\end{prop}
\begin{proof}
The number $B(m)$ is simply the least common multiple of the lengths of the periodic cycles of the maps $x \mapsto x^M$ on $(\pm 1)^n$ as $M$ ranges through $\GL_n(\QQ) \cap \operatorname{Mat}_{n \times n}(\NN)$.

The function $\alpha$ which conjugates $\Phi$ to a monomial map $\lambda \cdot x^M$ is definable in $\RR_{an}$, at least when restricted to some neighborhood of the origin, by the very definition of $\RR_{an}$.  Using Lemma~\ref{pm1} we may assume that each component of $\lambda$ is $\pm 1$ and from our choice of $B$ we then have that $\lambda^{M^B} = \lambda$.   Hence, for $a$ close enough to the origin we have $$\Phi^{\circ B n} (a) = \alpha^{-1} ( \lambda \cdot ( \alpha(a))^{M^{Bn}})$$  We define $$F(x) := \alpha^{-1}(\lambda \cdot (\alpha(a))^{E_m(x,M^B)})$$ where $E_m$ is the function of Lemma~\ref{expgln}.  We need to say a little about how to compute $(\alpha(a))^{(M^B)^x}$.   Write $(\alpha(a)) = (b_1,\ldots,b_n)$, then by our choice of $B$ we know that the sign of $b_i$ is the same as that of the $i^\text{th}$ component of $(\alpha(a))^{M^B}$.  Thus, we may compute $(\alpha(a))^{{M^B}^x}$ as $$(\frac{b_1}{|b_1|} \exp(\sum_{j=1}^n (E_m(x,M^B))_{1,j} \ln |b_j|), \ldots, \frac{b_n}{|b_n|} \exp(\sum_{j=1}^n (E_m(x,M^B))_{n,j} \ln |b_j|))$$ as long as $b_i \neq 0$ for all $i \leq n$.  If some $b_i = 0$, then each term in which $b_i$ appears will be zero and the remaining terms may expressed in the requisite form.
\end{proof}

\begin{cor}
\label{productexp}
If $\Phi$ is expressible as a Cartesian product of a strongly projectively linearizable function and a strongly monomializable function and $a$ is close enough to the origin, then possibly after replacing $\Phi$ by a compositional power there is a function $G:[0,\infty) \to \RR^n$ definable in $\RR_{an,\exp}$ satisfying $G(0) = a$ and $G(x+1) = \Phi(G(x))$.
\end{cor}
\begin{proof}
Write $\Phi = \Psi \times \Theta$ where $\Psi$ is strongly monomializable and $\Theta$ is strongly projectively linearizable.  Write $a = (b,c)$ relative to the decomposition of $\Phi$. Conjugating by an analytic function, we may assume that $\Theta$ is actually projectively linear.  By Proposition~\ref{monoexp} there is a definable function $F$ so that $F(0) = b$ and $F(x+1) = \Psi(F(x))$.  Using $F$ and Lemma~\ref{expgln} we may write $G(x) = (F(x),E_n(x+1,\Theta) c)$ where $n$ is the dimension of the linear part.
\end{proof}

Let us observe that Corollary~\ref{productexp} applies to functions expressed as products of univariate functions near their fixed points.

\begin{fact}
If $\Phi$ is a real analytic function in one variable near the origin which fixes the origin, then $\Phi^{\circ 2}$ is strongly projectively linearizable or strongly monomializable.
\end{fact}
\begin{proof}
We must break into the following cases:
\begin{itemize}
\item $\Phi \equiv 0$,
\item $\Phi'(0) = 0$ but $\Phi \not \equiv 0$,
\item $|\Phi'(0)| \neq 1$, and
\item $|\Phi'(0)| = 1$.
\end{itemize}

In the first case, $\Phi$ is obviously linear already.  The second case is an instance of B\"{o}ttcher's theorem~\cite{Boe} (see Chapter 9 of~\cite{Mi}).  The proof presented in~\cite{Mi} applies to complex analytic functions and the conclusion is stronger, namely that $\Phi$ is analytically conjugate to $x^N$ for where $N$ is its order of vanishing.  However, to conjugate $\Phi$ to $x^N$ it might be necessary to take an $(N-1)^\text{th}$ root which might not be possible over $\RR$.  The third case is a theorem of K{\oe}nigs~\cite{Koe} (see Chapter 8 of~\cite{Mi}).  The final case, of the so-called indifferent fixed points, is in general the most complicated case to study, but as we are dealing with real analytic functions of a single variable the only possibilities for $\Phi'(0)$ are $\pm 1$ and in this case $0$ is a parabolic fixed point for $\Phi^{\circ 2}$ and we may apply Leau's linearization theorem~\cite{Leau}  (Theorem 10.9 in~\cite{Mi}) $\Phi^{\circ 2}$ is conjugate to $x + 1$ at $\infty$ which when conjugated back to the origin is a projectively linear map.

In this one dimensional case, the monomializable functions are automatically strongly monomializable while the linearizable functions might not be strongly linearizable, but their compositional squares always are.
\end{proof}

\begin{thm}
If $\Phi$ is a real analytic function for which some positive compositional power is expressible as a product of a strongly projectively linearizable function and a strongly monomializable function and $a$ is close enough to the origin, then for any closed real analytic variety $X$ the set $\{ n \in \NN : \Phi^{\circ n} (a) \in X \}$ is a finite union of points and arithmetic progressions.  In particular, this result holds for $\Phi$ expressible as a product of univariate functions.
\end{thm}
\begin{proof}
Let $N \in \ZZ_+$ so that $\Phi^{\circ N}$ is a product of a strongly projectively linearizable function by a strongly monomializable function.  By Corollary~\ref{productexp} we find $\RR_{an,\exp}$ definable functions $G_0, \ldots, G_{N-1}$ so that $G_j(m) = \Phi^{\circ(Nm + j)}(a)$ for $m \in \NN$.  Each of the sets $\{x \in \RR : G_j(x) \in X \}$ is definable in $\RR_{an,\exp}$ and as such is a finite union of points and intervals.  Thus, the set of $m \in \NN$ with $\Phi^{\circ m}(a) \in X$ is a finite union of points and arithmetic progressions with modulus $N$.
\end{proof}

\section{Concluding speculations}
\label{cr}

We end this note with two observations about possible extensions of these methods.

First, it might seem that by interpreting $\CC$ as $\RR^2$ and complex analytic functions, via their real and imaginary parts, as real analytic functions in more variables, we could in a similar manner parametrize the orbits of complex analytic dynamical systems by $\RR_{an,\exp}$ definable functions near their fixed points.  However, we needed to restrict to \emph{strongly} projectively linearizable or \emph{strongly} monomializable functions.  Even if $\Phi$ is linearizable, if its eigenvalues are not real, then the orbits do not admit definable parametrizations.  However, they do admit reasonably concrete parametrizations and it might be reasonable to hope that a direct analytic argument not dependent on the theory of o-minimality could be employed.  On the other hand, it will happen in some cases that the orbit will not be globally parametrized, but it will be contained in the image of a definable function.  In this case, the o-minimal argument applies.   For example, this can happen near an irrationally indifferent point at which the self-map is linearizable and all of the eigenvalues of the corresponding linear map share a positive real period.

Miller has studied~\cite{Miller} the model theory of an expansion of an o-minimal structure on the real field by a trajectory for a definable vector field on the plane and has proven a kind of dichotomy theorem between tame behavior (d-minimality, every one variable definable set is a finite union of discrete sets and intervals) and wild behavior in which the integers are definable.  Some of the curves enveloping orbits in the complex analytic situation fall into his tame framework.  In related work, Miller and Tyne~\cite{MT} have shown that the structure obtained by naming an orbit of a definable unary function under the hypothesis that the orbit escapes to infinity and that the iterates of the function in question are cofinal in the set of all definable functions is also d-minimal.  These tameness theorems should generalize to functions of several variables and they should yield arithmetic information in the cases not amenable to the o-minimal analysis.

Secondly, we have used the o-minimality of $\RR_{an,\exp}$ rather crudely invoking merely its definition so as to say something about definable sets in one variable.   We might wish to study finitely many self-maps $\Phi_1,\ldots,\Phi_n$, points $a_1,\ldots,a_n$ and analytic sets $X$ (in the appropriate number of variables) and then look at the set
$$S := \{(m_1,\ldots,m_n) \in \NN^n : (\Phi_1^{\circ m_1}(a_1),\ldots,\Phi_n^{\circ m_n}(a_n)) \in X \}$$
If the orbits of $a_i$ under $\Phi_i$ admit $\RR_{an,\exp}$ definable parametrizations $F_i:[0,\infty) \to \RR^{m_i}$, then $S$ may be seen as the integer points on the $\RR_{an,\exp}$ definable set $$\widetilde{S} :=  \{ (x_1,\ldots,x_n) \in [0,\infty)^n : (F_1(x_1),\ldots,F_n(x_n)) \in X \}$$
In all generality, the set $S$ might contain several infinite families, but the theorems of Pila and Wilkie on counting rational points in sets definable in o-minimal structures limit the number integer points in $\widetilde{S}$ of small height.  More precisely, one must first compute the \emph{algebraic part}, $\widetilde{S}^{alg}$, of $\widetilde{S}$ by which we mean the union of all positive dimensional connected semialgebraic subsets of $\widetilde{S}$.  In practice, one expects that $\widetilde{S}^{alg}$ will be defined by some linear conditions, but the actual determination is a subtle problem.  Then one knows that for each $\epsilon > 0$ there is a constant $C(\epsilon)$ so that for $B \geq 1$
$$\# \{ (m_1,\ldots,m_n) \in S \smallsetminus \widetilde{S}^{alg} : |m_i| \leq B \text{ all } i \leq n \} \leq C(\epsilon) B^\epsilon$$

These bounds are much weaker than what one expects to be true, but bounds of any kind are notoriously difficult to obtain for the higher rank dynamical Mordell-Lang problem.

\bibliographystyle{plain}
\bibliography{eusk}

\end{document}